\documentclass{article}

\usepackage{amsmath,amscd,amsfonts,amssymb,amsthm}
\usepackage{graphicx}
\usepackage{a4wide}

\usepackage{epstopdf}

\newtheorem{thm}{Theorem}[section]
\newtheorem{proposition}[thm]{Proposition}
\newtheorem{definition}[thm]{Definition}
\newtheorem{remark}[thm]{Remark}
\newtheorem{example}{Example}

\def\a{\alpha}
\def\t{\tau}

\def\LDa{{_aD_t^\a}}
\def\LDaT{{_a\widetilde{D}_t^\a}}
\def\RD{{_tD_b^\a}}

\def\LDz{{_0D_t^{0.5}}}
\def\GLa{{^{GL}_{\phantom{1}a}D_t^\a}}
\def\GLb{{^{GL}_{\phantom{1}t}D_b^\a}}
\def\sGLa{{^{sGL}_{\phantom{1}a}D_t^\a}}
\def\w{\left(\omega_k^\a\right)}
\def\wh{\left(\omega_k^{0.5}\right)}

\newenvironment{keywords}{
\begin{center}
\begin{minipage}[c]{13.4cm}
{\bf Keywords:}}
{\end{minipage}
\end{center}}

\newenvironment{msc}{
\begin{center}
\begin{minipage}[c]{13.4cm}
{\bf 2010 Mathematics Subject Classification:}}
{\end{minipage}
\end{center}}

\begin{document}

\title{\text{Discrete Direct Methods in the Fractional Calculus of Variations}\footnote{This work
was partially presented 16-May-2012 by Shakoor Pooseh at FDA'2012,
who received a 'Best Oral Presentation Award'. Part of first author's Ph.D.,
which is carried out at the University of Aveiro under the Doctoral Program
in Mathematics and Applications (PDMA) of Universities of Aveiro and Minho.\newline
Submitted 26-Aug-2012; revised 25-Jan-2013; accepted 29-Jan-2013; 
for publication in \emph{Computers and Mathematics with Applications}.}}

\author{Shakoor Pooseh\\
\texttt{spooseh@ua.pt}
\and Ricardo Almeida\\
\texttt{ricardo.almeida@ua.pt}
\and Delfim F. M. Torres\\
\texttt{delfim@ua.pt}}

\date{CIDMA -- Center for Research and Development in Mathematics and Applications,\\
Department of Mathematics, University of Aveiro, 3810-193 Aveiro, Portugal}

\maketitle


\begin{abstract}
Finite differences, as a subclass of direct methods in the calculus of variations,
consist in discretizing the objective functional using appropriate approximations
for derivatives that appear in the problem. This article generalizes the same idea
for fractional variational problems. We consider a minimization problem with a Lagrangian
that depends on the left Riemann--Liouville fractional derivative.
Using the Gr\"{u}nwald--Letnikov definition, we approximate the objective functional
in an equispaced grid as a multi-variable function of the values
of the unknown function on mesh points. The problem is then transformed
to an ordinary static optimization problem. The solution to the latter problem gives
an approximation to the original fractional problem on mesh points.
\end{abstract}


\begin{msc}
26A33, 49M25.
\end{msc}


\section{Introduction}

Recently, fractional calculus, a classical branch of mathematical analysis
that studies non-integer powers of differentiation operators, has proved
a huge potential in solving complicated problems from science and engineering
\cite{Ref01,Ref02,Ref03}. In this framework,
the fractional calculus of variations is a research area under
strong current development. For the state of the art,
we refer the reader to the recent book \cite{book:frac},
for models and numerical methods we refer to \cite{book:Baleanu}.
A fractional variational problem consists in finding the extremizer
of a functional that depends on fractional derivatives and/or integrals
subject to some boundary conditions and possibly some extra constraints.
In this work we consider the following minimization problem:
\begin{equation}
\label{MainProblem}
\begin{gathered}
J[x(\cdot)]=\int_a^b L(t, x(t), \LDa x(t))dt \longrightarrow \min,\\
x(a)=x_a, \quad x(b)=x_b,
\end{gathered}
\end{equation}
that depends on the left Riemann--Liouville derivative,
$\LDa$, which is defined as follows.

\begin{definition}[see, e.g., \cite{Kilbas}]
Let $x(\cdot)$ be an absolutely continuous function in $[a,b]$ and $0\leq \a<1$.
The left Riemann--Liouville fractional derivative of order $\alpha$, $\LDa$,
is given by
\begin{equation*}
\LDa x(t)=\frac{1}{\Gamma(1-\alpha)}\frac{d}{dt}
\int_a^t (t-\t)^{-\alpha}x(\t)d\t, \quad t\in [a,b].
\end{equation*}
\end{definition}

There are several different definitions for fractional derivatives,
left and right, that can be found in the literature and could also be included.
They posses different properties: each one of those definitions
has its own advantages and disadvantages. Under certain conditions
they are, however, equivalent and can be used interchangeably.

There are two major approaches in the classical theory of calculus of variations.
In one hand, using Euler--Lagrange necessary optimality conditions, we can reduce
a variational problem to the study of a differential equation. Hereafter,
one can use either analytical or numerical methods to solve the differential
equation and reach the solution of the original problem (see, e.g., \cite{Kirk}).
This approach is referred as indirect methods in the literature. On the other hand,
we can tackle the functional itself, directly. Direct methods are used
to find the extremizer of a functional in two ways: Euler's finite differences
and Ritz methods. In the Ritz method, we either restrict admissible functions
to all possible linear combinations
$x_n=\sum_{i=1}^{n}\a_i\phi_i(t)$,
with constant coefficients $\a_i$ and a set of known base functions $\phi_i$,
or we approximate the admissible functions with such combinations. Using $x_n$
and its derivatives whenever needed, one can transform the functional
to a multivariate function of unknown coefficients $\a_i$.
By finite differences, however, we consider the admissible functions
not on the class of arbitrary curves, but only on polygonal curves
made upon a given grid on the time horizon. Using an appropriate
discrete approximation of the Lagrangian, and substituting the integral with a sum,
we can transform the main problem to the optimization of a function of several parameters:
the values of the unknown function on mesh points (see, e.g., \cite{Elsgolts}).

Indirect methods for fractional variational problems have a vast background
in the literature and can be considered a well studied subject: see
\cite{Agrawal,pooseh:2012,Atan1,Frederico,Jelicic,Klimek,MyID:227,Riewe97}
and references therein that study different variants of the problem and discuss
a bunch of possibilities in the presence of fractional terms, Euler--Lagrange
equations and boundary conditions. Direct methods, however, to the best of our knowledge,
have got less interest and are not well studied. A brief introduction
of using finite differences has been made in \cite{Riewe96}, which can be regarded
as a predecessor to what we call here an Euler-like direct method.
A generalization of Leitmann's direct method can be found in \cite{AlD1},
while \cite{Lotfi} discusses the Ritz direct method for optimal control problems
that can easily be reduced to a problem of the calculus of variations.


\section{Direct methods in the classical theory}

The basic idea of direct methods is to consider the variational problem
as a limiting case of a certain extremum problem of a multi-variable function.
This is then an ordinary static optimization problem. The solution
to the latter problem can be regarded as an approximate solution
to the original variational problem. There are three major methods:
Euler's finite differences, Ritz and Kantorovich's methods.
We are going to discuss the Euler method briefly: see, e.g., \cite{Elsgolts}.
The basic idea of a finite differences method is that instead
of considering the values of a functional
\begin{equation*}
J[x(\cdot)]=\int_a^b L(t, x(t), \dot{x}(t))dt
\end{equation*}
with boundary conditions $x(a)=x_a$ and $ x(b)=x_b$,
on arbitrary admissible curves, we only track the values at an $n+1$ points grid,
$t_i$, $i=0,\ldots,n$, in the interested time interval. The functional $J[x(\cdot)]$
is then transformed into a function $\Psi(x(t_1),x(t_2),\ldots,x(t_{n-1}))$
of the values of the unknown function on mesh points. Assuming
$h=t_{i}-t_{i-1}$, $x(t_i)=x_i$ and $\dot{x}_i\simeq \frac{x_{i}-x_{i-1}}{h}$, one has
\begin{gather*}
J[x(\cdot)] \simeq \Psi(x_1,x_2,\ldots,x_{n-1})
= h\sum_{i=1}^{n}L\left(t_i, x_i,\frac{x_{i}-x_{i-1}}{h}\right),\\
x_0=x_a, \quad x_n=x_b.
\end{gather*}
The desired values of $x_i$, $i=1,\ldots,n-1$, give the extremum
to the multi-variable function $\Psi$ and satisfy the system
$$
\frac{\partial \Psi}{\partial x_i}=0,\quad i=1,\ldots,n-1.
$$

The fact that only two terms in the sum, $(i-1)$th and $i$th,
depend on $x_i$ makes it rather easy to find the extremum of $\Psi$
solving a system of algebraic equations. For each $n$, we obtain
a polygonal line, which is an approximate solution to the original problem.
It has been shown that passing to the limit, as $h\rightarrow 0$,
the linear system corresponding to finding the extremum of $\Psi$
is equivalent to the Euler--Lagrange equation for the original problem \cite{Tuckey}.


\section{Finite differences for fractional derivatives}

In classical theory, given a derivative of certain order $x^{(n)}$,
there is a finite difference approximation of the form
\begin{equation*}
x^{(n)}(t)=\lim_{h\rightarrow 0^+}
\frac{1}{h^n} \sum_{k=0}^n(-1)^k\binom{n}{k}x(t-kh),
\end{equation*}
where
\begin{equation}
\label{binom}
\binom{n}{k}=\frac{n(n-1)(n-2)\cdots (n-k+1)}{k!}.
\end{equation}
This is generalized to derivatives of arbitrary order
and gives rise to the Gr\"{u}nwald--Letnikov fractional derivative.

\begin{definition}[see, e.g., \cite{Kilbas}]
Let $0<\a<1$. The left Gr\"{u}nwald--Letnikov fractional derivative is defined as
\begin{equation}
\label{LGLdef}
\GLa x(t)=\lim_{h\rightarrow 0^+} \frac{1}{h^\a}
\sum_{k=0}^\infty(-1)^k\binom{\a}{k}x(t-kh).
\end{equation}
Here $\binom{\a}{k}$ is the generalization of binomial coefficients \eqref{binom} to real numbers.
Similarly, the right Gr\"{u}nwald--Letnikov derivative is given by
\begin{equation}
\label{RGLdef}
\GLb x(t)=\lim_{h\rightarrow 0^+} \frac{1}{h^\a}
\sum_{k=0}^\infty(-1)^k\binom{\a}{k}x(t+kh).
\end{equation}
\end{definition}

The series in \eqref{LGLdef} and \eqref{RGLdef} converges absolutely and uniformly
if $x(\cdot)$ is bounded. The definition is clearly affected by the non-local property
of fractional derivatives. The arbitrary order derivative of a function at a time $t$
depends on all values of that function in $(-\infty,t]$ and $[t,\infty)$ because
of the infinite sum, backward and forward difference nature of the left and right derivatives,
respectively. Since we are, usually, and specifically in this paper, dealing with closed time intervals,
the following remark is made to make the definition clear in closed regions.

\begin{remark}
\label{RemFinite}
For the above definition to be consistent, we need the values of $x(t)$
outside the interval $[a,b]$. To overcome this difficulty, we take
\begin{equation*}
x^*(t)
=
\begin{cases}
x(t) & \text{ if } \ t\in [a,b],\\
0 & \text{ if } \ t\notin [a,b].
\end{cases}
\end{equation*}
Thus, one can assume $\GLa x(t)=\GLa x^*(t)$
and $\GLb x(t)=\GLb x^*(t)$ for $t\in [a,b]$.
\end{remark}

As we mentioned before, this definition coincides with Riemann--Liouville derivatives.
The following proposition establishes the connection between these two definitions
and that of Caputo derivative, another type of derivative that is believed
to be more applicable in practical fields such as engineering and physics.

\begin{proposition}[see, e.g., \cite{Podlubny}]
Let $n-1< \a<n$, $n\in \mathbb{N}$, and $x(\cdot)\in C^{n-1}[a,b]$.
Suppose also that $x^{(n)}(\cdot)$ is integrable on $[a,b]$. Then the
Riemann--Liouville derivative exists and coincides with the
Gr\"{u}nwald--Letnikov:
\begin{equation*}
\LDa x(t) = \sum_{i=0}^{n-1}\frac{x^{(i)}(a)(t-a)^{i-\a}}{\Gamma(1+i-\a)}
+\frac{1}{\Gamma(n-\a)}\int_a^t(t-\t)^{n-1-\a}x^{(n)}(\t)d\t\\
=\GLa x(t).
\end{equation*}
\end{proposition}

\begin{remark}
For numerical purposes, we need a finite sum in \eqref{LGLdef}
and this goal is achieved by Remark~\ref{RemFinite}.
Given a grid on $[a,b]$ as $a=t_0,t_1,\ldots,t_n=b$,
where $t_i=t_0+ih$ for some $h>0$,
we approximate the left Riemann--Liouville derivative as
\begin{equation}
\label{GLApprx}
\LDa x(t_i)\simeq \frac{1}{h^\a} \sum_{k=0}^{i}\w x(t_i-kh)=:\LDaT x(t_i),
\end{equation}
where $\w=(-1)^k\binom{\a}{k}=\frac{\Gamma(k-\a)}{\Gamma(-\a)\Gamma(k+1)}$.
Similarly, one can approximate the right Riemann--Liouville  derivative by
\begin{equation}
\label{GLApprxR}
\RD x(t_i)\simeq \frac{1}{h^\a} \sum_{k=0}^{n-i}\w x(t_i+kh).
\end{equation}
As it is stated in \cite{Podlubny}, this approximation is of first order, i.e.,
\begin{equation*}
\LDa x(t_i)= \frac{1}{h^\a} \sum_{k=0}^{i}\w x(t_i-kh)+\mathcal{O}(h).
\end{equation*}
\end{remark}

\begin{remark}
In \cite{Meerschaert}, it has been shown that the implicit Euler method
is unstable for certain fractional partial differential equations
with Gr\"{u}nwald--Letnikov approximations. Therefore, discretizing
fractional derivatives, shifted Gr\"{u}nwald--Letnikov derivatives are used.
Despite the slight difference with respect to standard Gr\"{u}nwald--Letnikov
derivatives, they exhibit, at least for certain cases, a stable performance.
The shifted Gr\"{u}nwald--Letnikov derivative is defined by
\begin{equation*}
\sGLa x(t_i)\simeq \frac{1}{h^\a} \sum_{k=0}^{i}\w x(t_i-(k-1)h).
\end{equation*}
\end{remark}

Other finite difference approximations can be found in the literature,
e.g., Diethelm's backward finite difference formulas
for fractional derivatives \cite{Kai2}.


\section{Euler-like direct method for fractional variational problems}
\label{DirMethSec}

As mentioned earlier, we consider a simple version of the fractional
variational problem where the fractional term consists of a Riemann--Liouville
derivative on a finite time interval $[a,b]$. The boundary conditions are given
and we approximate the fractional derivative
using the Gr\"{u}nwald--Letnikov approximation given by \eqref{GLApprx}.
In this context we discretize the functional in \eqref{MainProblem}
using a simple quadrature rule on the mesh points $a=t_0, t_1, \ldots, t_n=b$
with $h=\frac{b-a}{n}$. The goal is to find the values $x_1, x_2, \ldots, x_{n-1}$
of the unknown function $x(\cdot)$ at points $t_i$, $i=1,2,\ldots,n-1$.
The values of $x_0$ and $x_n$ are given. Applying the quadrature rule gives
\begin{equation*}
J[x(\cdot)] = \sum_{i=1}^{n}\int_{t_{i-1}}^{t_{i}} L(t_i, x_i, \LDa x_i)dt
\simeq \sum_{i=1}^{n} h L(t_i, x_i, \LDa x_i)
\end{equation*}
and by approximating the fractional derivatives at mesh points using \eqref{GLApprx} we have
\begin{equation}
\label{disFuncl}
J[x(\cdot)]\simeq \sum_{i=1}^{n}hL\left(t_i, x_i, \LDaT x_{i-k}\right).
\end{equation}
Hereafter the procedure is the same as in classical case.
The right hand side of \eqref{disFuncl} can be regarded
as a function $\Psi$ of $n-1$ unknowns
$\mathbf{x}=(x_1,x_2,\ldots,x_{n-1})$,
\begin{equation}
\label{sumFrac}
\Psi(\mathbf{x})=\sum_{i=1}^{n}hL\left(t_i, x_i, \LDaT x_{i-k}\right).
\end{equation}
To find an extremum for $\Psi$, one has to solve
the following system of algebraic equations:
\begin{equation}
\label{AlgSys}
\frac{\partial \Psi}{\partial x_i}=0,\qquad i=1,\ldots,n-1.
\end{equation}
Unlike the classical case, all terms in \eqref{sumFrac}, starting from the $i$th term,
depend on $x_i$ and we have
\begin{equation}
\label{GenSys}
\frac{\partial \Psi}{\partial x_i}
=h\frac{\partial L}{\partial x}(t_i,x_i,\LDaT x_i)
+\sum_{k=0}^{n-i}\frac{h}{h^\a}\w
\frac{\partial L}{\partial\LDa x}(t_{i+k},x_{i+k},\LDaT x_{i+k}).
\end{equation}
Equating the right hand side of \eqref{GenSys} with zero one has
\begin{align*}
\frac{\partial L}{\partial x}(t_i,x_i,\LDaT x_i)
+\frac{1}{h^\a}\sum_{k=0}^{n-i}\w\frac{\partial L}{\partial
\LDa x}(t_{i+k},x_{i+k},\LDaT x_{i+k})=0.
\end{align*}
Passing to the limit and considering the approximation formula
for the right Riemann--Liouville derivative \eqref{GLApprxR},
we can prove the following result.

\begin{thm}
The Euler-like method for a fractional variational problem
of the form \eqref{MainProblem} is equivalent
to the fractional Euler--Lagrange equation
\begin{equation*}
\frac{\partial L}{\partial x}+\RD\frac{\partial L}{\partial\LDa x}=0,
\end{equation*}
as the mesh size, $h$, tends to zero.
\end{thm}

\begin{proof}
Consider a minimizer $(x_1,\ldots,x_{n-1})$ of $\Psi$, a variation function $\eta\in C[a,b]$
with $\eta(a)=\eta(b)=0$ and define $\eta_i=\eta(t_i)$, for $i=0,\ldots,n$.
We remark that $\eta_0=\eta_n=0$ and that $(x_1+\epsilon\eta_1,\ldots,x_{n-1}+\epsilon\eta_{n-1})$
is a variation of $(x_1,\ldots,x_{n-1})$, with $|\epsilon|<r$, for some fixed $r>0$.
Therefore, since $(x_1,\ldots,x_{n-1})$ is a minimizer for $\Psi$,
proceeding with Taylor's expansion, we deduce that
\begin{multline*}
0 \leq \Psi(x_1+\epsilon\eta_1,\ldots,x_{n-1}
+\epsilon\eta_{n-1})-\Psi(x_1,\ldots,x_{n-1})\\
=\epsilon\sum_{i=1}^n h\left[ \frac{\partial L}{\partial x}[i]\eta_i
+\frac{\partial L}{\partial {_aD_t^{\a}}}[i] \frac{1}{h^\a}
\sum_{k=0}^i(\omega^\a_k) \eta_{i-k} \right]+\mathcal{O}(\epsilon),
\end{multline*}
where here, and in what follows,
$$
[i]=\left(t_i,x_i,\frac{1}{h^\a} \sum_{k=0}^i(\omega^\a_k) x_{i-k} \right).
$$
Since $\epsilon$ takes any value, we can write that
\begin{equation}
\label{sum1}
\sum_{i=1}^n h\left[ \frac{\partial L}{\partial x}[i]\eta_i
+\frac{\partial L}{\partial {_aD_t^{\a}}}[i]
\frac{1}{h^\a}\sum_{k=0}^i(\omega^\a_k) \eta_{i-k} \right]=0.
\end{equation}
On the other hand, since $\eta_0=0$, reordering the terms of the sum,
it follows immediately that
$$
\sum_{i=1}^n \frac{\partial L}{\partial {_aD_t^{\a}}}[i]
\sum_{k=0}^i(\omega^\a_k) \eta_{i-k}= \sum_{i=1}^n \eta_i
\sum_{k=0}^{n-i}(\omega^\a_k)  \frac{\partial L}{\partial {_aD_t^{\a}}}[i+k].
$$
Substituting this relation into equation \eqref{sum1}, we obtain
$$
\sum_{i=1}^n\eta_i h\left[ \frac{\partial L}{\partial x}[i]
+\frac{1}{h^\a} \sum_{k=0}^{n-i}(\omega^\a_k)\frac{\partial
L}{\partial {_aD_t^{\a}}}[i+k] \right]=0.
$$
Since $\eta_i$ is arbitrary, for $i=1,\ldots,n-1$, we deduce that
$$
\frac{\partial L}{\partial x}[i]+\frac{1}{h^\a}
\sum_{k=0}^{n-i}(\omega^\a_k)\frac{\partial L}{\partial
{_aD_t^{\a}}}[i+k]=0, \quad i=1,\ldots,n-1.
$$
Let us study the case when $n$ goes to infinity. Let $\overline{t} \in ]a,b[$
and $i\in\{1,\ldots,n\}$ be such that $t_{i-1}<\overline{t} \leq t_i$. First
observe that in such case, we also have $i\to\infty$ and $n-i\to\infty$.
In fact, let $i\in\{1,\ldots,n\}$ be such that
$a+(i-1)h<\overline{t}\leq a+ih$.
Then, $i<(\overline{t}-a)/h+1$, which implies that
$$
n-i>n\frac{b-\overline{t}}{b-a}-1.
$$
Therefore,
$$
\lim_{n\to\infty,i\to\infty}t_i=\overline{t}.
$$
Assume that there exists a function $\overline{x} \in C[a,b]$ satisfying
$$
\forall \epsilon>0\, \exists N \, \forall n\geq N \,:
|x_i-\overline{x}(t_i)|<\epsilon, \quad \forall i=1,\ldots,n-1.
$$
As $\overline{x}$ is uniformly continuous, we have
$$
\forall \epsilon>0\, \exists N \, \forall n\geq N \,:
|x_i-\overline{x}(\overline{t})|<\epsilon, \quad \forall i=1,\ldots,n-1.
$$
By the continuity assumption of $\overline{x}$, we deduce that
$$
\lim_{n\to\infty,i\to\infty}\frac{1}{h^\a} \sum_{k=0}^{n-i}(\omega^\a_k)\frac{\partial
L}{\partial {_aD_t^{\a}}}[i+k]={_tD^{\a}_b }\frac{\partial L}{\partial
{_aD_t^{\a}}}(\overline{t}, \overline{x} (\overline{t}),{_aD_{\overline{t}}^{\a}}
\overline{x}(\overline{t})).
$$
For $n$ sufficiently large (and therefore $i$ also sufficiently large),
$$
\lim_{n\to\infty,i\to\infty}\frac{\partial L}{\partial x}[i]
=\frac{\partial L}{\partial x}(\overline{t}, \overline{x}(\overline{t}),
{_aD_{\overline{t}}^{\a}} \overline{x}(\overline{t})).
$$
In conclusion,
\begin{equation}
\label{fracELE}
\frac{\partial L}{\partial x}(\overline{t}, \overline{x}(\overline{t}),
{_aD_{\overline{t}}^{\a}} \overline{x}(\overline{t}))+{_tD^{\a}_b }\frac{\partial L}{\partial
{_aD_t^{\a}}}(\overline{t}, \overline{x}(\overline{t}),{_aD_{\overline{t}}^{\a}} \overline{x}(\overline{t}))=0.
\end{equation}
Using the continuity condition, we prove that the fractional Euler--Lagrange equation \eqref{fracELE}
holds for all values on the closed interval $a\leq t\leq b$.
\end{proof}


\section{Examples}

In this section we try to solve test problems using what has been presented
in Section~\ref{DirMethSec}. For the sake of simplicity, we restrict ourselves
to the interval $[0,1]$. To measure the accuracy of our method we use the maximum norm.
Assume that the exact value of the function $x(\cdot)$, at the point $t_i$, is $x(t_i)$
and it is approximated by $\tilde{x}_i$. The error is defined as
\begin{equation}
\label{Error}
E=\max\{|x(t_i)-\tilde{x}_i|, \quad i=1,\ldots,n-1\}.
\end{equation}

\begin{example}
\label{Example1}
Consider the following minimization problem:
\begin{equation}
\label{Exmp1}
\begin{gathered}
J[x(\cdot)]=\int_0^1 \left(\LDz x(t)-\frac{2}{\Gamma(2.5)}t^{1.5}\right)^2 dt \longrightarrow \min,\\
x(0)=0, \quad x(1)=1.
\end{gathered}
\end{equation}
Problem \eqref{Exmp1} has the obvious solution of the form $\hat{x}(t)=t^2$
due to the positivity of the Lagrangian and the zero value of $J[\hat{x}(\cdot)]$.
Using the approximation
\begin{equation*}
\LDz x(t_i)\simeq \frac{1}{h^{0.5}} \sum_{k=0}^{i}\wh x(t_i-kh)
\end{equation*}
for a fixed $h$, and following the routine discussed
in Section~\ref{DirMethSec}, we approximate problem \eqref{Exmp1} by
\begin{equation*}
\Psi(\mathbf{x})=\sum_{i=1}^{n}h\left(\frac{1}{h^{0.5}}
\sum_{k=0}^{i}\wh x_{i-k}-\frac{2}{\Gamma(2.5)}t^{1.5}_i\right)^2.
\end{equation*}
Since the Lagrangian in this example is quadratic, system \eqref{AlgSys}
is linear and therefore easy to solve. Other problems may end with
a system of nonlinear equations. Simple calculations lead to
\begin{equation}
\label{Ex1Sys}
\mathbf{A}\mathbf{x}=\mathbf{b},
\end{equation}
in which
\begin{equation*}
\mathbf{A}=\left[\begin{array}{llll}
\sum_{i=0}^{n-1}A_i^2        & \sum_{i=1}^{n-1}A_{i}A_{i-1} & \cdots & \sum_{i=n-2}^{n-1}A_{i}A_{i-n+2} \\
\sum_{i=0}^{n-2}A_{i}A_{i+1} & \sum_{i=1}^{n-2}A_i^2        & \cdots & \sum_{i=n-3}^{n-2}A_{i}A_{i-n+3} \\
\sum_{i=0}^{n-3}A_{i}A_{i+2} & \sum_{i=1}^{n-3}A_{i}A_{i+1} & \cdots & \sum_{i=n-4}^{n-3}A_{i}A_{i-n+4} \\
\vdots                       & \vdots                       & \ddots & \vdots                        \\
\sum_{i=0}^{1}A_{i}A_{i+n-2} & \sum_{i=0}^{1}A_{i}A_{i+n-3} & \cdots & \sum_{i=0}^{1}A_i^2
\end{array}\right],
\end{equation*}
where $A_i=(-1)^{i}h^{1.5}\binom{0.5}{i}$, $\mathbf{x}=(x_1,x_2,\cdots,x_{n-1})^T$
and $\mathbf{b}=(b_1,b_2,\cdots,b_{n-1})^T$ with
\begin{equation*}
b_i=\sum_{k=0}^{n-i}\frac{2h^2A_k}{\Gamma(2.5)}t_{k+i}^{1.5}-A_{n-i}A_0-\left(\sum_{k=0}^{n-i}A_kA_{k+i}\right).
\end{equation*}
The linear system \eqref{Ex1Sys} is easily solved for different values of $n$.
As illustrated in Figure~\ref{Ex1Fig}, by increasing the value of $n$ we get better approximations.
\begin{figure}[!tp]
\begin{center}
\includegraphics[width=8cm]{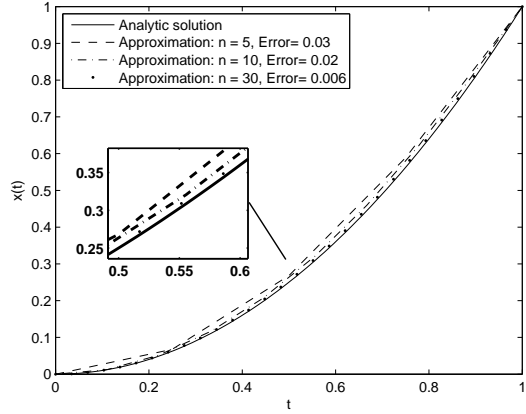}
\end{center}
\caption{Analytic and approximate solutions for
problem \eqref{Exmp1} of Example~\ref{Example1}.}\label{Ex1Fig}
\end{figure}
\end{example}

Although we constructed our theory for problems in the form \eqref{MainProblem},
other operators can be included in the Lagrangian. Let us now move to another example
that depends also on the first derivative and the solution
is obtained via the fractional Euler--Lagrange equation.
Suppose that the objective functional in the fractional variational problem
depends on the left Riemann--Liouville fractional derivative, $\LDa$,
and on the first derivative, $\dot{x}$. If $x(\cdot)$ is an extremizer
of such a problem, then it satisfies the following
fractional Euler--Lagrange equation (for a proof see \cite{odz:2012b}):
\begin{equation}
\label{ELT}
\frac{\partial L}{\partial x}+\RD \frac{\partial L}{\partial \LDa x}
-\frac{d}{dt}\left(\frac{\partial L}{\partial \dot{x}}\right)=0.
\end{equation}

\begin{example}
\label{Example2}
Consider the following minimization problem:
\begin{equation}
\label{Exmp2}
\begin{gathered}
J[x(\cdot)]=\int_0^1 \left(\LDz x(t)-\dot{x}^2(t)\right) dt \longrightarrow \min,\\
x(0)=0, \quad x(1)=1.
\end{gathered}
\end{equation}
In this case the Euler--Lagrange equation \eqref{ELT} gives
$_tD_1^{0.5} 1+2\ddot{x}(t)=0$.
Since $_tD_1^{0.5}1=\frac{(1-t)^{-0.5}}{\Gamma(0.5)}$,
the fractional Euler--Lagrange equation turns
to be an ordinary differential equation:
$$
\ddot{x}(t)=-\frac{1}{2\Gamma(0.5)}(1-t)^{-0.5},
$$
which subject to the given boundary conditions has solution
\begin{equation*}
x(t)=-\frac{1}{2\Gamma(2.5)}(1-t)^{1.5}
+\left(1-\frac{1}{2\Gamma(2.5)}\right)t+\frac{1}{2\Gamma(2.5)}.
\end{equation*}
Discretizing problem \eqref{Exmp2}, with the same assumptions
of Example~\ref{Example1}, ends in the linear system
\begin{equation}
\label{Ex2Sys}
\left[\begin{array}{ccccccc}
2      & -1    & 0     & 0    & \cdots & 0 & 0\\
-1     & 2     & -1    & 0    & \cdots & 0 & 0\\
0      & -1    & 2     & -1   & \cdots & 0 & 0\\
\vdots & \vdots& \vdots&\vdots& \ddots & \vdots&\vdots \\
0      &  0    & 0     &  0   & \cdots &-1 & 2
\end{array}\right]\left[\begin{array}{c}x_1\\x_2\\x_3\\\vdots\\x_{n-1}\end{array}\right]
= \left[\begin{array}{c}b_1\\b_2\\b_3\\\vdots\\b_{n-1}\end{array}\right],
\end{equation}
where
$$
b_i=\frac{h}{2}\sum_{k=0}^{n-i-1}(-1)^{k}h^{0.5}\binom{0.5}{k},
\qquad i=1,2,\ldots,n-2,
$$
and
$$
b_{n-1}=\frac{h}{2}\sum_{k=0}^{1}\left((-1)^{k}h^{0.5}\binom{0.5}{k}\right)+x_n.
$$
The linear system \eqref{Ex2Sys} can be solved efficiently for any $n$ to reach
the desired accuracy. The analytic solution together with some approximations
are shown in Figure~\ref{Ex2Fig}.
\begin{figure}[!tp]
\begin{center}
\includegraphics[width=8cm]{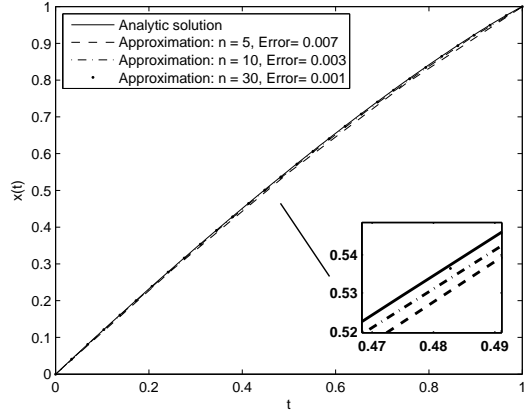}
\end{center}
\caption{Analytic and approximate solutions for problem \eqref{Exmp2}
of Example~\ref{Example2}.}\label{Ex2Fig}
\end{figure}
\end{example}

Both Examples~\ref{Example1} and \ref{Example2} end with linear systems,
and their solvability is simply dependent to the matrix of coefficients.
Now we try our method on a more complicated problem,
yet analytically solvable with an oscillating solution.

\begin{example}
\label{Example3}
Consider the minimization problem \eqref{MainProblem}
with the Lagrangian
\begin{equation}
\label{eq:lag:ex3}
L=\left(\LDz x(t)-\frac{16\Gamma(6)}{\Gamma(5.5)}t^{4.5}+
\frac{20\Gamma(4)}{\Gamma(3.5)}t^{2.5}-\frac{5}{\Gamma(1.5)}t^{0.5}\right)^4
\end{equation}
and subject to the boundary conditions $x(0)=0$ and $x(1)=1$.
The functional $\int_0^1 Ldt$, with nonnegative $L$,
attains its minimum value for
\begin{equation}
\label{eq:exact:sol:ex3}
x(t)=16t^{5}-20t^{3}+5t.
\end{equation}
The appearance of the fourth power in the Lagrangian \eqref{eq:lag:ex3},
results in a nonlinear system when we apply the Euler-like direct method
to this problem. For $j=1,2,\ldots,n-1$ we have
\begin{equation}
\label{Sys3}
\sum_{i=j}^n \left(\omega_{i-j}^{0.5}\right)\left(\frac{1}{h^{0.5}}
\sum_{k=0}^{i}\wh x_{i-k}-\phi(t_i)\right)^3=0,
\end{equation}
where
$$
\phi(t)=\frac{16\Gamma(6)}{\Gamma(5.5)}t^{4.5}
+\frac{20\Gamma(4)}{\Gamma(3.5)}t^{2.5}-\frac{5}{\Gamma(1.5)}t^{0.5}.
$$
System \eqref{Sys3} is solved for different values of $n$
and the results are given in Figure~\ref{Ex3Fig}, where we compare
the obtained approximations with the exact solution \eqref{eq:exact:sol:ex3}.
\begin{figure}[!htp]
\begin{center}
\includegraphics[width=8cm]{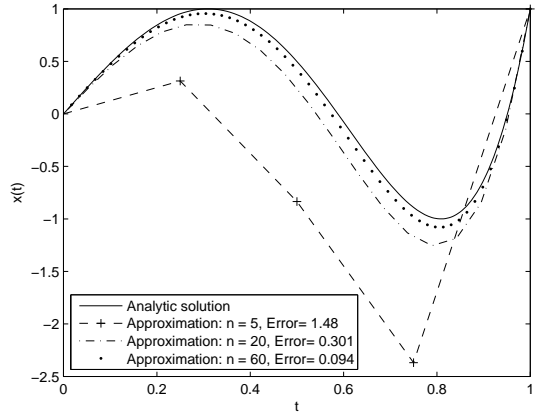}
\end{center}
\caption{Analytic and approximate solutions for problem
of Example~\ref{Example3}.}\label{Ex3Fig}
\end{figure}
\end{example}


\section{Conclusion}

Roughly speaking, an Euler-like direct method reduces a variational problem
to the solution of a system of algebraic equations. When the system is linear,
we can freely increase the number of mesh points, $n$, and obtain better solutions
as long as the resulted matrix of coefficients is invertible. The method
is very fast in this case and the execution time is of order $10^{-4}$
for Examples~\ref{Example1} and \ref{Example2}. It is worth, however, to keep in mind
that the  Gr\"{u}nwald--Letnikov approximation is of first order, $\mathcal{O}(h)$,
and even a large $n$ can not result in a high precision. Actually, by increasing $n$,
the solution slowly converges and in Example~\ref{Example2}, a grid of $30$ points
has the same order of error, $10^{-3}$, as a $5$ points grid.
The situation is completely different when the problem ends with a nonlinear system.
In Example~\ref{Example3}, a small number of mesh points, $n=5$, results in a poor
solution with the error $E=1.4787$. The Matlab built in function \textsf{fsolve}
takes $0.0126$ seconds to solve the problem. As one increases the number of mesh points,
the solution gets closer to the analytic solution and the required time increases drastically.
Finally, by $n=90$ we have $E=0.0618$ and the time is $T=26.355$ seconds.
Table~\ref{table} summarizes the results.
\begin{table}[!htp]
\center
\begin{tabular}{|c|c|c|c|}
\hline
          &  n &  T   & E \\
\hline
Example 1 &  5  &  $1.9668\times 10^{-4}$ & 0.0264 \\
          &  10 &  $2.8297\times 10^{-4}$ & 0.0158 \\
          &  30 &  $9.8318\times 10^{-4}$ & 0.0065 \\
\hline
Example 2 &  5  &  $2.4053\times 10^{-4}$ & 0.0070 \\
          &  10 &  $3.0209\times 10^{-4}$ & 0.0035 \\
          &  30 &  $7.3457\times 10^{-4}$ & 0.0012 \\
\hline
Example 3 &  5  &  0.0126      & 1.4787 \\
          &  20 &  0.2012      & 0.3006 \\
          &  90 &  26.355      & 0.0618 \\
\hline
\end{tabular}
\caption{Number of mesh points, $n$, with corresponding run time in seconds,
$T$, and error, $E$ \eqref{Error}.}\label{table}
\end{table}
In practice, we have no idea about the solution in advance and the worst case should
be taken into account. Comparing the results of the three examples considered,
reveals that for a typical fractional variational problem,
the Euler-like direct method needs a large number of mesh points
and most likely a long running time.

The Euler-like direct method for fractional variational problems here presented
can be improved in some stages. One can try different approximations
for the fractional derivative that exhibit higher order precisions.
Better quadrature rules can be applied to discretize the functional and, finally,
we can apply more sophisticated algorithms for solving the resulting system of algebraic equations.
Further works are needed to cover different types of fractional variational problems.


\section*{Acknowledgments}

Work partially presented at The 5th Symposium
on Fractional Differentiation and its Applications (FDA'12),
held in Hohai University, Nanjing, May 14-17, 2012.
Supported by {\it FEDER} funds through
{\it COMPETE} --- Operational Programme Factors of Competitiveness
(``Programa Operacional Factores de Competitividade'')
and by Portuguese funds through the
{\it Center for Research and Development
in Mathematics and Applications} (University of Aveiro)
and the Portuguese Foundation for Science and Technology
(``FCT --- Funda\c{c}\~{a}o para a Ci\^{e}ncia e a Tecnologia''),
within project PEst-C/MAT/UI4106/2011
with COMPETE number FCOMP-01-0124-FEDER-022690.
Pooseh was also supported by FCT through the Ph.D. fellowship
SFRH/BD/33761/2009.



\end{document}